%% file: main.tex
\documentclass{siamart190516}

\input{ex_shared}

\begin{document}

\maketitle

\begin{abstract}
    This article presents a randomized matrix-free method for approximating the trace of $f(\mat{A})$, where $\mat{A}$ is a large symmetric matrix and $f$ is a function analytic in a closed interval containing the eigenvalues of $\mat{A}$. Our method uses a combination of stochastic trace estimation (i.e., Hutchinson's method), Chebyshev approximation, and multilevel Monte Carlo techniques. We establish general bounds on the approximation error of this method by extending an existing error bound for Hutchinson's method to multilevel trace estimators. Numerical experiments are conducted for common applications such as estimating the log-determinant, nuclear norm, and Estrada index, and triangle counting in graphs. We find that using multilevel techniques can substantially reduce the variance of existing single-level estimators. 
\end{abstract}



\section{Introduction}
Given a symmetric matrix $\mat{A}\in \mathbb{R}^{d\times d}$ and a function $f: \mathbb{R}\rightarrow\mathbb{R}$, we consider the problem of estimating 
\begin{equation} \label{def:trace}
    \trace \left(f(\mat{A})\right) = \sum_{i=1}^d f(\lambda_i), 
\end{equation}
where $\lambda_1,\ldots,\lambda_d$ are the eigenvalues of $\mat{A}$. This could in theory be done by computing the eigenvalues of $\mat{A}$, but when $\mat{A}$ is large this option is impractical. A cheaper option is to use stochastic trace estimation, which estimates $\trace \left(f(\mat{A})\right)$ by computing quantities of the form $\mat{z}\ts f(\mat{A})\mat{z}$, where $\mat{z}$ is a random vector. 

Four functions of particular interest are $f(x) = \log (x)$, $f(x) = 1/x$, $f(x) = \exp(x)$, and $f(x) = x^{p/2}$, which correspond respectively to the log-determinant of a matrix, the trace of the inverse, the Estrada index, and the Schatten $p$-norm\footnote{In the latter case, we use $\|\mat{X}\|_p^p = \trace (\mat{X}\ts\mat{X})^{p/2} = \trace f(\mat{A})$, where $\mat{A} = \mat{X}\ts \mat{X}$.}. For these functions it is not practical to compute $\mat{z}\ts f(\mat{A})\mat{z}$ to machine precision, but neither is it necessary for the purpose of estimating the quantity in \eqref{def:trace}. Instead, it suffices to estimate $\mat{z}\ts f(\mat{A})\mat{z}$ by constructing a polynomial or rational approximation to $f$, or by using Lanczos quadrature \cite{ubaru2017fast,golub2009matrices}. The accuracy, and therefore the cost, of these approximations is governed by the accuracy to which one wishes to estimate $\trace \left(f(\mat{A})\right)$. A typical analysis of one of these methods might provide a theorem along the following lines: 

\begin{quote}
    In order to estimate $\trace \left(f(\mat{A})\right)$ to tolerance $\varepsilon$ with failure probability at most $\delta$, sample $\mat{z}\ts f(\mat{A})\mat{z}$ at least $m$ times with a level-$n$ approximation of $f(\mat{A})$. 
\end{quote}
In the above, the term ``level-$n$'' may refer to a degree-$n$ polynomial approximation or an $n$-point quadrature rule\textemdash either way, larger values of $n$ correspond to more accurate and expensive approximations. 

The aim of this article is to provide a general mechanism by which such methods might be improved. 

\subsection{Our approach}
We propose a method for reducing the cost of any stochastic trace estimation technique that approximates quantities of the form $\mat{z}\ts f(\mat{A})\mat{z}$ to variable accuracy. We focus specifically on Chebyshev approximation, but the method may be adapted to use Taylor series, rational approximations, or Lanczos quadrature. It may also be used in conjunction with other variance reduction methods such as those in \cite{meyer2021hutch++}. 

The key idea is that by taking many samples with a crude approximation to $f(\mat{A})$ and a few samples with an accurate approximation to $f(\mat{A})$, we can obtain a better estimate than we would have gotten simply by taking a moderate number of samples with an accurate approximation. This technique is known as {\it multilevel Monte Carlo} \cite{giles2008multilevel}, which was originally developed for path simulation problems and has since found a wide variety of applications including chemical reaction networks \cite{anderson2012multilevel}, aerospace engineering \cite{geraci2017multifidelity}, and rare event estimation \cite{ullmann2015multilevel}. Our application of multilevel techniques to trace estimation is outlined in Section \ref{sec:multilevel}.

In applying multilevel techniques to trace estimation problems, the user must choose how to set the levels: how crude should a ``crude'' approximation to $f(\mat{A})$ be, and how many different approximations should be used? Under a certain framework it turns out that these questions have an optimal answer, summarized by Theorem \ref{thm:levels}. Based on this theorem, we propose a method for selecting the levels automatically based on a pilot sample. 

We also show that existing error bounds for trace estimation using Hutchinson's method may be extended to multilevel methods. This result is presented in Theorem \ref{thm:dethm}, which offers a general framework for deriving $(\delta,\epsilon)$-type error guarantees for multilevel estimators. 

Numerical experiments show that the multilevel estimator can have a significantly smaller variance than the single-level estimator, particularly on nuclear norm estimation problems. We also consider the problem of triangle counting in graphs, and show that using a certain set of control variates can modestly reduce the variance of existing trace estimates at minimal additional cost. 

\subsection{Summary of contributions}
The key contributions of this article are as follows. 
\begin{itemize}
    \item Equations \eqref{def:qlevels} and \eqref{eqn:decomposition} show how multilevel Monte Carlo techniques may be applied to stochastic trace estimation problems. 
    \item Using Theorem \ref{thm:levels}, we propose a method for selecting levels automatically and without the need for additional user input. 
    \item Theorem \ref{thm:dethm} extends existing error bounds for single-level trace estimators to the multilevel framework. 
    \item In Section \ref{sec:triangle} we propose a related variance reduction technique for estimating the number of triangles in a graph. 
    \item Numerical experiments in Section \ref{sec:experiments} demonstrate the practical benefit of our methods. 
\end{itemize}

\subsection{Outline}
Section \ref{sec:background} provides background on stochastic trace estimation, Chebyshev interpolation, and multilevel Monte Carlo methods. Section \ref{sec:multilevel} describes how multilevel methods may be applied to trace estimation and provides a procedure for selecting the parameters for the multilevel estimator. Section \ref{sec:errbounds} generalizes an existing error bound for single-level estimators to the multilevel case. Section \ref{sec:experiments} contains the results of numerical experiments on real data, and Section \ref{sec:conclusion} offers our concluding remarks. 

\subsection{Notation}
Matrices, vectors, integers, and scalars will typically be denoted as $\mat{A}$, $\mat{a}$, $a$, and $\alpha$, respectively, with $\mat{I}$ denoting the identity matrix. The expressions $\mathbb{E}[X]$ and $\mathbb{V}[X]$ respectively denote the expected value and variance of a random variable $X$. 
The trace of a matrix $\mat{A}$ is $\trace(\mat{A})$ and $\|\mat{A}\|_F$ and $\|\mat{A}\|_2$ are its Frobenius and operator norms, respectively. If $\mat{A}\in\mathbb{R}^{d\times d}$ is a symmetric matrix with spectral decomposition $\sum_{i=1}^d\lambda_i\mat{q}_i\mat{q}_i\ts$, then for a real-valued function $f$ we define $f(\mat{A}) = \sum_{i=1}^df(\lambda_i)\mat{q}_i\mat{q}_i\ts$. 

\section{Background} \label{sec:background}
Here we review Hutchinson's method, Chebyshev approximation, and multilevel Monte Carlo methods. For more background on these topics, see e.g. \cite{han2017approximating} and \cite{giles2015multilevel}. 

\subsection{Hutchinson's method}
A common method for trace estimation relies on the following theorem \cite{hutchinson1989stochastic}: 
\begin{theorem}
Let $\mat{A} \in \mathbb{R}^{d\times d}$ be a symmetric matrix, and let $\mat{z}\in \mathbb{R}^d$ be a random variable such that $\mathbb{E}[\mat{z}\mat{z}\ts] = \mat{I}$. Then \[\E[\mat{z}\ts f(\mat{A})\mat{z}] = \trace \left(f(\mat{A})\right).\]
\end{theorem}
If we generate random samples $\mat{z}^{(1)},\ldots,\mat{z}^{(m)}$ from a Rademacher distribution (entries $\pm 1$ with equal probability), the {\it Hutchinson estimator} is then given by 
\begin{equation}\label{def:mcEstimate}
    \Gamma_m = \frac{1}{m}\sum_{i=1}^{m} \mat{z}^{(i)T} f(\mat{A}) \mat{z}^{(i)}. 
\end{equation}
Ideally, we will be able to compute or estimate each term $\mat{z}^{(i)T} f(\mat{A})\mat{z}^{(i)}$ with only a small number of matrix-vector products with $\mat{A}$. Thus if $\mat{A}$ is sparse or otherwise permits fast matrix-vector multiplication, the estimator in \eqref{def:mcEstimate} will be cheap to compute. 

\begin{example}
If $\mat{A}$ is the $\{0,1\}$-valued adjacency matrix for a graph, the number of triangles in the graph is equal to $\frac{1}{6}\trace(\mat{A}^3)$. Each term of the form $\mat{z}^{(i)T} \mat{A}^3\mat{z}^{(i)}$ may be evaluated using only three matrix-vector products\footnote{Or two, if the symmetry of the quadratic form is exploited.}, and so stochastic trace estimation allows us to estimate the number of triangles in a graph without having to compute $\mat{A}^3$ explicitly. 
\end{example}

In the case where $f(\mat{A})$ is symmetric positive semi-definite (SPSD), the following error guarantee for the Hutchinson estimator is derived in \cite{roosta2015improved}:
\begin{theorem}[Roosta-Khorasani/Ascher]\label{thm:hutch}
Let $\mat{A}$ be SPSD. For a given pair $(\varepsilon,\delta)$ of positive numbers, the bound 
    \[|\Gamma_m - \trace (\mat{A})|\leq \varepsilon \trace (\mat{A})\]
    holds with failure probability at most $\delta$ if $ m \geq 6\varepsilon^{-2}\ln(2/\delta)$.
\end{theorem}

\subsection{Chebyshev interpolation}
We consider Chebyshev polynomials of the first kind, which follow the recurrence relation 
\begin{equation}\label{eqn:recurrence}
    T_{j+1}(x) = 2xT_j(x) - T_{j-1}(x), \quad j = 1,2,\ldots
\end{equation}
with $T_0(x) = 1$ and $T_1(x) = x$. A given function $f:[-1,1]\rightarrow\mathbb{R}$ can then be approximated by the degree-$n$ interpolating polynomial
\begin{equation}\label{eqn:chebSeries}
    f(x) \approx p_n(x) = \sum_{j=0}^n c_{j}T_j(x). 
\end{equation}
The interpolating nodes for $p_n$ are given\footnote{Other variants exist; see \cite{trefethen2008gauss} for details.} by
\begin{equation}
    x_{j} = \cos \frac{j\pi}{n}, \quad 0\leq j \leq n, 
\end{equation}
and the coefficients $c_{j}$ can be elegantly computed using a fast Fourier transform \cite{trefethen2008gauss}. 

Since $p_n$ interpolates $f$ on the interval $[-1,1]$, it follows that $p_n(\mat{A})$ will be a good approximation to $f(\mat{A})$ if the spectrum of $\mat{A}$ lies in the interval $[-1,1]$. For a matrix whose spectrum lies in $[a,b]$, we can find an affine function $g$ that maps $[a,b]$ to $[-1,1]$, define 
\[\tilde{f} = f\circ g^{-1}, \quad \widetilde{\mat{A}} = g(\mat{A}),\]
and approximate $\trace (\tilde{f}(\widetilde{\mat{A}}))$ using a Chebyshev interpolation of $\tilde{f}$. We can therefore assume without loss of generality that the spectrum of $\mat{A}$ is contained in $[-1,1]$, although doing so requires at least a rough estimate of the maximum and minimum eigenvalues of $\mat{A}$.  

Expressions of the form $\mat{z}\ts p_n(\mat{A})\mat{z}$ can then be evaluated by computing $\mat{z}_n = p_n(\mat{A})\mat{z}$ using the recurrence in \eqref{eqn:recurrence}, then returning $\mat{z}\ts \mat{z}_n$. Details can be found in \cite{han2017approximating}, and the process requires $n$ matrix-vector products (matvecs) with $\mat{A}$. It is observed in \cite{hallman2021faster} that by exploiting the symmetry of the quadratic form the number of matvecs can be reduced to $\lceil n/2\rceil$.

\subsection{Multilevel Monte Carlo}
Given a sequence $P_1,\ldots,P_{L-1}$ of random variables approximating $P_L$ with increasing accuracy, the quantity of interest $\E[P_L]$ can be rewritten as the telescoping sum
\begin{equation}\label{eqn:telescope}
    \E[P_L] = P_1 + \sum_{k=2}^L \E[P_k - P_{k-1}].
\end{equation}
We can then estimate $\E[P_L]$ by estimating each term on the right hand side independently. The insight of multilevel Monte Carlo methods is that if the low-level terms in \eqref{eqn:telescope} are cheap to compute and the high-level terms have small variance, this strategy can be more efficient than sampling $P_L$ alone \cite{giles2008multilevel,giles2015multilevel}.  

Let $C_1$ and $m_1$ denote the cost of computing a single sample of $P_1$ and the number of times it was sampled, and let $V_1 = \var[P_1]$. Similarly, for $2\leq k \leq L$ let $C_k$ and $m_k$ respectively denote the cost of a single sample of $P_k - P_{k-1}$ and the number of times it was sampled, and let $V_k = \var[P_k - P_{k-1}]$. 
For a fixed variance $\varepsilon^2$, the total cost $C$ is minimized by setting 
\begin{align}
    m_k &= \mu \sqrt{V_k/C_k}, \quad \text{where} \label{eqn:optSamples}\\
    \mu &= \varepsilon^{-2}\sum_{k=1}^L\sqrt{V_k C_k}. \label{eqn:optMu}
\end{align}
The total cost of the estimate is therefore given by
\begin{equation}\label{eqn:totalCost}
    C = \sum_{k = 1}^L m_k C_k 
    =
    \varepsilon^{-2}\left(\sum_{k=0}^L\sqrt{V_k C_k}\right)^2. 
\end{equation}
If a computational budget $C$ is prescribed rather than a target variance, then the formula for $\mu$ will change but equations \eqref{eqn:optSamples} and \eqref{eqn:totalCost} will still hold. Note that the solution in \eqref{eqn:optSamples} is only an approximation as it allows the terms $m_k$ to take on non-integer values. 

Giles \cite{giles2015multilevel} observes that if the terms $V_k C_k$ are generally decreasing with $k$ then the first term $V_1C_1$ will make the largest contribution to the overall cost. If this is the case, we can hope to attain a total cost along the lines of $C\approx \varepsilon^{-2}V_1C_1$, as opposed to the standard cost $C\approx \varepsilon^{-2}V_1C_L$ that would come from estimating $P_L$ alone.

\section{Multilevel trace estimation} \label{sec:multilevel}
The form of the interpolating polynomial in \eqref{eqn:chebSeries} suggests a natural way to apply multilevel techniques to Chebyshev approximation. For a fixed degree $n$ for the interpolant $p_n$ and indices $-1\leq \ell' < \ell \leq n$, we define the variables 
\begin{equation} \label{def:qlevels}
        Q_{\ell'\ell} = \sum_{j=\ell'+1}^{\ell} c_j \mat{z}\ts T_j(\mat{A})\mat{z}.
\end{equation}
Given a sequence $0\leq \ell_1 < \ell_2< \cdots < \ell_L = n$ we obtain the decomposition
\begin{equation} \label{eqn:decomposition}
    \mat{z}\ts p_n(\mat{A})\mat{z} = \sum_{k=1}^L Q_{\ell_{k-1}\ell_k},
\end{equation}
where for convenience we will always take $\ell_0$ to be equal to $-1$. Thus a choice of levels $\{\ell_k\}_{k=1}^L$ corresponds to a partition of $\mat{z}\ts p_n(\mat{A})\mat{z}$ into $L$ parts, each of which is a sum of consecutive terms in the polynomial. 

Altering the notation of the previous section somewhat, we define $V_{\ell'\ell}$ and $C_{\ell'\ell}$ to be the variance and cost of estimating $Q_{\ell'\ell}$. The basic framework for the multilevel method is then as follows: we choose a set of levels $\{\ell_k\}_{k=1}^L$, then take a pilot sample to estimate the variance at each level. Given a desired variance $\varepsilon^2$ or computational budget $C$, we then use \eqref{eqn:optSamples} and \eqref{eqn:optMu} to determine the optimal number of samples $m_k$ for each level.

\subsection{Cost estimates}\label{sec:costs}
For Chebyshev interpolation, the cost of a sample will be more or less proportional to the number of matvecs required. The cost of sampling $Q_{\ell'\ell}$ can therefore be modeled as $\ell$ if we use the methods of \cite{han2017approximating}, or as $\lceil \ell/2\rceil$ if we exploit the symmetry of the quadratic form as in \cite{hallman2021faster}. 

\subsection{Optimal level selection}\label{sec:levelSelection}

Considering the form of the cost in \eqref{eqn:totalCost}, it is critical to note that using a large number of levels may be counterproductive, particularly if the corresponding variances decay slowly. A judicious choice of levels is therefore necessary if we want our multilevel method to outperform the single-level estimator. Here we present a method for choosing the levels with the aim of minimizing the total cost as given in \eqref{eqn:totalCost}. 

Recalling that the approximate cost of the multilevel method is given by \eqref{eqn:totalCost}, we define for $0\leq \ell \leq n$ the variables 
\begin{equation}\label{eqn:levelCost}
    \mathcal{C}_\ell := \min_{L}\min_{\{\ell_k\}}  \sum_{k=1}^L \sqrt{V_{\ell_{k-1}\ell_k}C_{\ell_{k-1}\ell_k}},
\end{equation}
where the minimization is taken over indices satisfying $0\leq \ell_1 < \cdots < \ell_L = \ell$. In particular, $\mathcal{C}_n$ corresponds to the optimal multilevel cost of approximating $\mat{z}\ts p_n(\mat{A})\mat{z}$. Our goal is to find the set of levels corresponding to this optimal cost. 

With perfect information about the variances and costs, it turns out that we can efficiently find the set of levels corresponding to $\mathcal{C}_n$ through dynamic programming. We summarize this finding in the form of the following theorem. 

\begin{theorem}\label{thm:levels}
  For $0\leq \ell\leq n$, let $\mathcal{C}_{\ell}$ be defined as in \eqref{eqn:levelCost}. Then $\mathcal{C}_n$ can be computed by the recurrence
\begin{equation}\label{eqn:costrecurrence}
    \mathcal{C}_{\ell} = \begin{cases}0 & \ell = 0, \\
    \min_{0\leq \ell' < \ell}\ \mathcal{C}_{\ell'} + \sqrt{V_{\ell'\ell}C_{\ell'\ell}} & 1\leq \ell \leq n. 
    \end{cases}
\end{equation}
\end{theorem}

Assuming we already know the variances $V_{\ell'\ell}$, Theorem \ref{thm:levels} implies that we can compute $\mathcal{C}_n$ in $\mathcal{O}(n^2)$ time. The optimal levels associated with $\mathcal{C}_n$ can be obtained at minimal extra cost. Since $n$ is small compared to the size of $\mat{A}$, determining the optimal levels will be inexpensive compared to the overall cost of trace estimation. 

\subsubsection{Application to Chebyshev interpolation}
In applying the level selection method of Theorem \ref{thm:levels} to Chebyshev interpolation, we face two complications. The first is that we do not have prior knowledge of the variances and so must estimate them. The second is that equations \eqref{eqn:optSamples} and \eqref{eqn:optMu} assume that the sample sizes $\{m_k\}_{k=1}^L$ may take on non-integer values. As a result, our method as described runs the risk of selecting too many levels and recommending $0< m_k \ll 1$ for the more expensive levels. 

We propose to estimate the variances by taking a pilot sample. For $1\leq i \leq m_\text{pilot}$ we compute the terms $\{c_j\mat{z}^{(i)T} T_j(\mat{A})\mat{z}^{(i)}\}_{j=0}^n$, storing them in a matrix of size $m_\text{pilot}\times (n+1)$. The variances can then be estimated from this information in $\mathcal{O}(n^2m_\text{pilot})$ time, which will generally be small compared to the overall cost of trace estimation. The pilot samples may subsequently be reused for the trace estimate. 

\begin{remark}
An alternate method might be to use the Chebyshev coefficients to bound the variances at each level. We tried but ultimately rejected this approach, as the resulting bounds were too pessimistic. Performance improved when we made the assumption that $\mat{z}\ts T_j(\mat{A})\mat{z}$ and $\mat{z}\ts T_{j'}(\mat{A})\mat{z}$ were uncorrelated whenever $j\neq j'$, but it is not clear whether this assumption is realistic enough to be reliable. 
\end{remark}


To resolve the issue of the sample sizes taking on non-integer values, we make the following modification: for $0\leq \ell \leq n-1$ we compute $\mathcal{C}_{\ell}$ using the recursion in \eqref{eqn:costrecurrence}, but when computing $\mathcal{C}_n$ we add the additional constraint that the number of samples recommended for the highest level should be at least $m_{\text{pilot}}$. To obtain a ``recommended'' number of samples, we require either a target variance $\varepsilon^2$ or a computational budget $C$. 

\section{Error bounds for multilevel methods} \label{sec:errbounds}
In this section, we derive error guarantees for multilevel methods. Given a set of levels $\{\ell_k\}_{k=1}^L$ and sample sizes $\mat{m} = \{m_k\}_{k=1}^L$, we define the estimator
\begin{equation}\label{eqn:multiEstimator}
    \Gamma_{\mat{m}} = \sum_{k = 1}^L \sum_{i=1}^{m_k} \frac{1}{m_k}Q_{\ell_{k-1}\ell_k}^{(i,k)},
\end{equation}
where the $(i,k)$ superscripts denote independent samples. Alternately, we may define for $1\leq k \leq L$ the matrices 
\begin{equation}\label{def:bmat}
    \mat{A}_k = \sum_{j=\ell_{k-1}+1}^{\ell_k}c_jT_j(\mat{A}).
\end{equation}
Then $p_n(\mat{A}) = \sum_{k=1}^L \mat{A}_k$, and we can express the multilevel estimator in the form
\begin{equation}\label{eqn:multiEstimator2}
    \Gamma_{\mat{m}} = \sum_{k = 1}^L \sum_{i=1}^{m_k} \frac{1}{m_k}\mat{z}^{(i,k)T}\mat{A}_k\mat{z}^{(i,k)},
\end{equation}
where the $\mat{z}^{(i,k)}$ are independently drawn Rademacher vectors. We can then obtain bounds on the accuracy of the estimator $\Gamma_{\mat{m}}$ by using the following theorem, due to \cite{cortinovis2020randomized}: 
\begin{theorem}[Cortinovis/Kressner] \label{thm:kressner}
    Let $\mat{z}\in\mathbb{R}^d$ be a Rademacher vector and let $\mat{A}\in\mathbb{R}^{d\times d}$ be a nonzero symmetric matrix with all-zero diagonal entries. Then for all $\varepsilon > 0$, 
    \begin{equation}\label{eqn:kressnerBound}
        \prob\left(|\mat{z}\ts \mat{A}\mat{z}|\geq \varepsilon\right) \leq 2\exp\left(-\frac{\varepsilon^2}{8\|\mat{A}\|_F^2 + 8\varepsilon\|\mat{A}\|_2}\right). 
    \end{equation}
\end{theorem}
Cortinovis and Kressner subsequently use Theorem \ref{thm:kressner} to derive error bounds for single-level estimates. We use the same proof technique to extend their bounds to multilevel methods. 

\begin{theorem}\label{thm:dethm}
    Let $\widehat{\mat{A}}\in\mathbb{R}^{d\times d}$ be a nonzero symmetric matrix. Let $\{\mat{A}_k\}_{k=1}^L$ be symmetric matrices such that $\widehat{\mat{A}} = \sum_{k=1}^L \mat{A}_k$, and for $1\leq k\leq L$ let $\mat{B}_k$ equal $\mat{A}_k$ but with the diagonal entries set to zero. For sample sizes $\mat{m} = \{m_k\}_{k=1}^L$, let $\Gamma_{\mat{m}}$ be defined as in \eqref{eqn:multiEstimator2}. Then for all $\varepsilon > 0$, 
    \begin{equation}\label{eqn:multiBound}
        \prob\left(|\Gamma_{\mat{m}} - \trace(\widehat{\mat{A}})|\geq \varepsilon\right) \leq 
        2\exp\left(\frac{-\varepsilon^2/8}
        {\sum_{k=1}^L\|\mat{B}_k\|_F^2/m_k + \varepsilon \displaystyle{\max_{1\leq k \leq L}}\|\mat{B}_k\|_2/m_k
        }\right).
    \end{equation}
    Furthermore, for $1\leq k \leq L$ let $V_k = \|\mat{B}_k\|_F^2 + \varepsilon\|\mat{B}_k\|_2$ and let $C_k$ represent the cost of sampling from $\mat{B}_k$. Then if $m_k \geq \mu\sqrt{V_k/C_k}$ where 
     \begin{equation}\label{eqn:demu}
             \mu = 8\varepsilon^{-2}\log(2/\delta)\sum_{k=1}^L\sqrt{V_kC_k},
     \end{equation}
    it follows that $\prob\left(|\Gamma_{\mat{m}} - \trace(\widehat{\mat{A}})|\geq \varepsilon\right) \leq \delta$.
\end{theorem}
\begin{proof}
    Let $m = \sum_{k=1}^L m_k$, and let $\mat{B}$ be a block diagonal matrix in $\mathbb{R}^{md\times md}$ with $m_k$ copies of $\mat{B}_k/m_k$ as its diagonal blocks. The matrix $\mat{B}$ has zero diagonal and satisfies 
    \begin{align*}
        \|\mat{B}\|_F^2 &= \sum_{k=1}^L\|\mat{B}_k\|_F^2/m_k, \\
        \|\mat{B}\|_2 &= \max_{1\leq k\leq L} \|\mat{B}_k\|_2/m_k. 
    \end{align*}
    The first result follows by applying Theorem \ref{thm:kressner} to $\mat{B}$. The second follows by using the relaxation $\max_{1\leq k \leq L}\|\mat{B}_k\|_2/m_k \leq \sum_{k=1}^L \|\mat{B}_k\|_2/m_k$ and setting the failure probability in \eqref{eqn:multiBound} to $\delta$. 
\end{proof}

When using the sampling strategy proposed in Theorem \ref{thm:dethm}, the total cost of the multilevel estimator for a given pair $(\varepsilon,\delta)$ can be approximated as
\begin{equation}
    C = 8\varepsilon^{-2}\log(2/\delta)\left(\sum_{k=1}^L\sqrt{V_kC_k}\right)^2.
\end{equation}
This expression closely resembles the one in \eqref{eqn:totalCost}, with the caveat that $V_k$ and $\varepsilon$ refer to different quantities in these two equation. The single-level estimator, by comparison, guarantees an error of $\varepsilon$ with failure probability $\delta$ at a cost of $8\varepsilon^{-2}V_{\text{tot}}C_{\text{tot}}$, where $V_{\text{tot}} = \|\mat{B}\|_F^2 + \varepsilon\|\mat{B}\|_2$ and $C_{\text{tot}}$ is the cost of sampling from $\mat{B}$. 

In short, multilevel estimators can be expected not only to have smaller variances than their single-level counterparts, but better $(\varepsilon,\delta)$-type error bounds as well. One limitation of Theorem \ref{thm:dethm} is that since we do not know $\|\mat{B}_k\|_F$ or $\|\mat{B}_k\|_2$ in advance, it does not directly give practical advice on how to choose the number of samples. More work must be done to derive error guarantees for individual functions of interest, as is done in \cite{han2017approximating} or \cite{ubaru2017fast}, but we leave this matter for a future study.

\section{Numerical experiments} \label{sec:experiments}
In this section we conduct several experiments to examine the behavior of our multilevel estimator, particularly in comparison to single-level methods. All experiments were conducted using MATLAB 2020b on an Intel Core i7 3.5GHz machine, and the code used to produce all figures and tables is available at \url{https://github.com/erhallma/multilevel-trace-estimation/}. 

Table \ref{fig:test_matrices} contains a list of the matrices used in our experiments. Most come from the SuiteSparse database \cite{suiteSparse}. Matrices ca-GrQc and wiki-Vote are the exceptions, which were obtained from the Stanford Large Network Dataset Collection\footnote{See \url{https://snap.stanford.edu/data/}.}. We examine four functions in particular: $f(x) = \sqrt{x}$ (for estimating the nuclear norm), $f(x) =\log (x)$ (log determinant), $f(x) =\exp(x)$ (Estrada index), and $f(x) =x^3$ (triangle counting). For information on practical applications, we refer the reader to \cite{ubaru2017applications} or \cite{han2017approximating} and the references therein. 

\begin{table}
    \centering
       \begin{tabular}{|l|c|c|c|}\hline
       Matrix  & Application & Size & nnz \\ \hline
       thermal2& Thermal & 1228045 & 8580313\\
       thermomechTC & Thermal & 102158 & 711558 \\
       boneS01 & Model reduction & 127224 & 5516602\\
       ecology2 & 2D/3D & 999999 & 4995991\\
       ukerbe1 & 2D/3D & 5981 & 15704 \\
       dictionary28 & undirected graph & 52652 &  178076\\
       Erdos02 & undirected graph & 6927 & 16944 \\
       fe\_4elt2 & undirected graph & 11143 & 65636 \\
        California & Web search & 9664 & 16150\\
        deter3 & Linear program & 7647$\times$21777 & 44547 \\
        FA & Pajek network & 10617 & 72176 \\
        Roget & Pajek network & 1022 & 7297 \\
        ca-GrQc & undirected graph & 5242 & 28968 \\
        wiki-Vote & undirected graph & 7115 & 201524 \\\hline
    \end{tabular}
    \caption{Matrices used in our numerical experiments. All matrices with the exception of deter3 are square.}
    \label{fig:test_matrices}
\end{table}

\subsection{Automated level selection}
In section \ref{sec:levelSelection} we propose a method for choosing the levels on the basis of a pilot sample and without the need for further user input. Here we illustrate how this method behaves in practice. 

We use our multilevel method to estimate the nuclear norm of the matrix FA (see Table \ref{fig:test_matrices}), representing a directed unweighted graph with $10617$ nodes and $72176$ edges. We estimate $\trace\left((\mat{A}\ts \mat{A})^{1/2}\right)$, approximating the function $f(x) = x^{1/2}$ with a degree 300 polynomial and using a budget of 15,000 matvecs, the equivalent of 50 samples for a single-level method. In order to estimate the variances at each level for the purpose of level selection, we tke 10 samples using the degree 300 approximation. 

Figure \ref{fig:FAbehavior} shows the behavior of the multilevel method over 100 trials. In general, we make the following observations: 
\begin{itemize}
    \item The number of levels chosen is highly variable. The median trial uses twenty levels, but the number of levels ranges from as few as seven to as many as fifty-nine. 
    \item The selected levels tend to appear in a smaller number of clusters. Table \ref{fig:level_example} shows one fairly typical case using eighteen levels. Aside from the consecutive levels 1--11, the method also selects the smaller clusters (45,46,47) and (63,65) in this example. 
    \item Despite the variability in the {\it number} of levels selected, the budget allocation by level is fairly consistent between trials. For example, a typical trial spends around 60-70 percent of its computational budget on polynomials of degree 50 or less.
    \item The median degree of the second most expensive level is 86 over the 100 trials, and the maximum degree is 137. Thus although a high-degree polynomial may be needed to obtain a certain approximation accuracy, the multilevel method devotes most of its effort to estimating terms of significantly lower degree. 
\end{itemize}

\begin{figure}[H]
    \centering
    	\begin{minipage}{0.49\textwidth}
		\includegraphics[width=\textwidth]{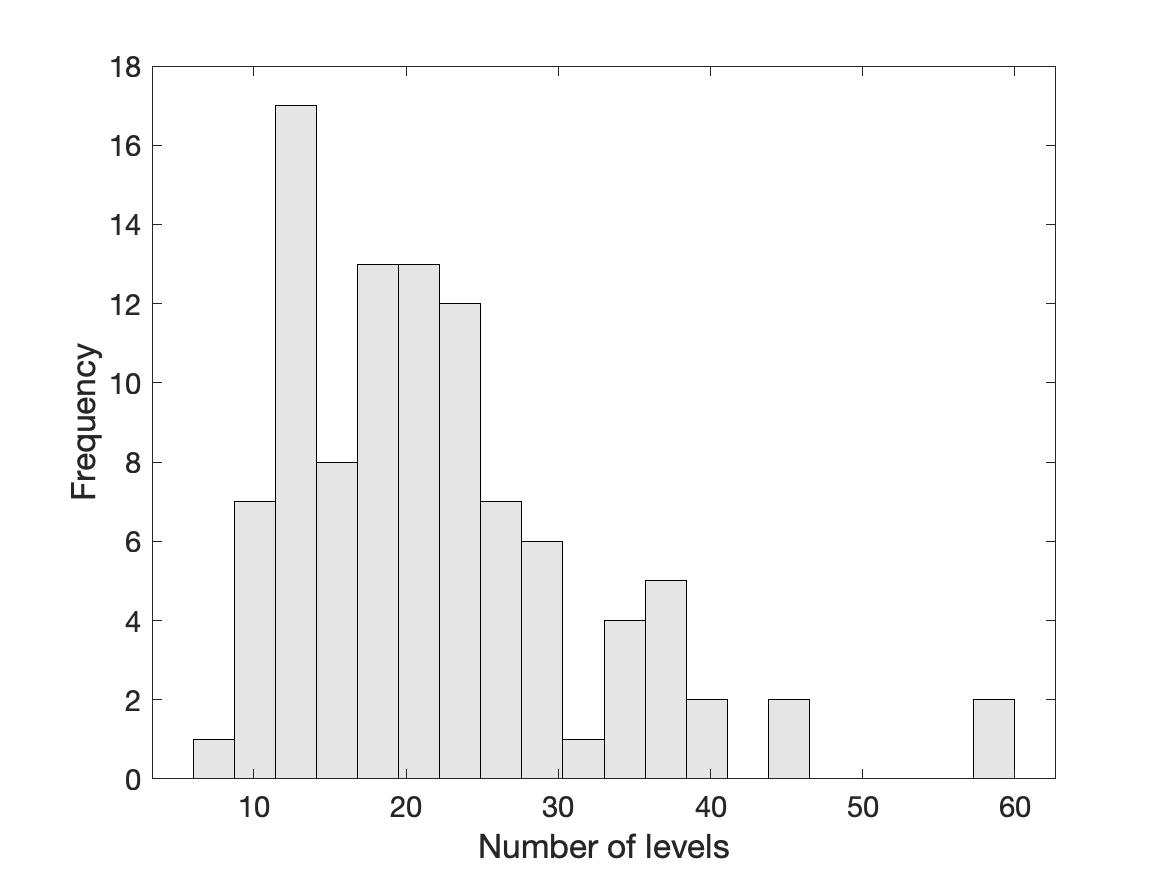}
	\end{minipage}
	\begin{minipage}{0.49\textwidth}
		\includegraphics[width=\textwidth]{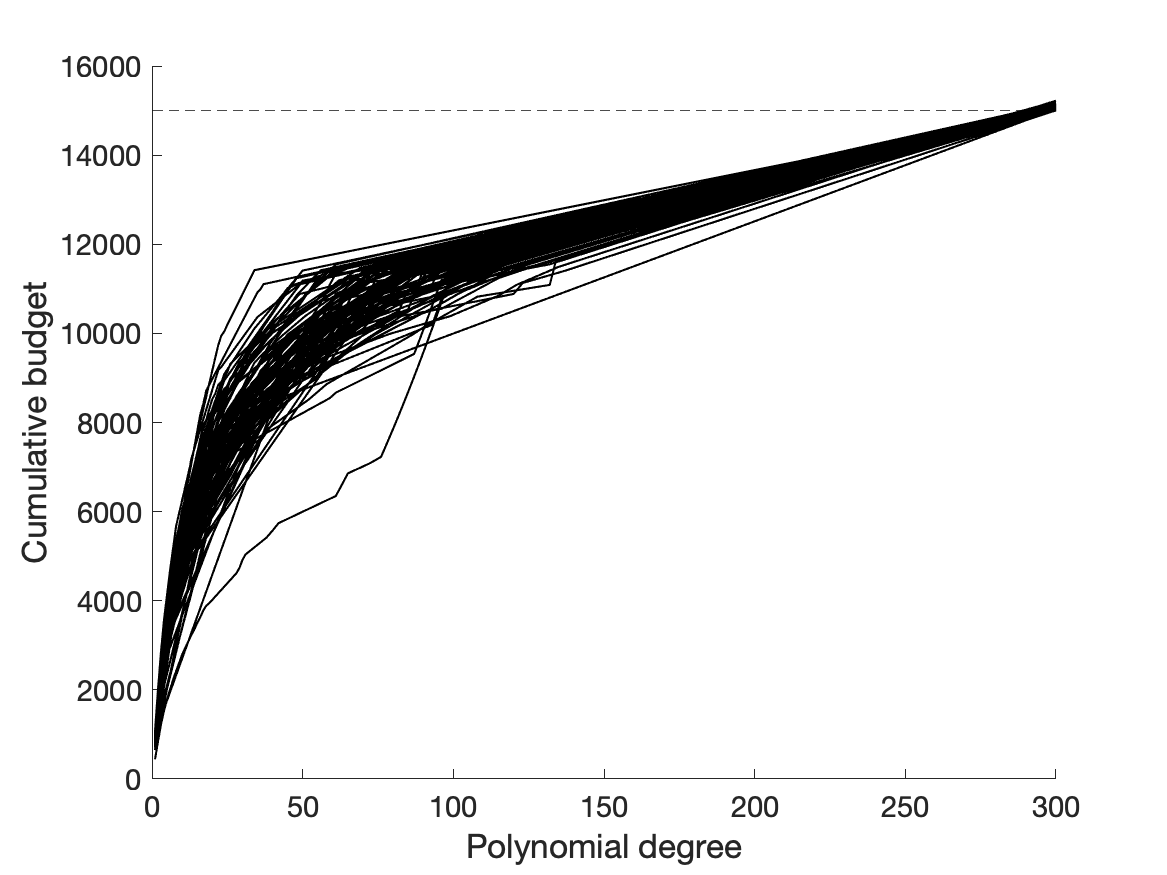}
	\end{minipage}
    \caption{Behavior of automated level selection over 100 trials. Left: number of levels chosen. Right: budget allocation. }
    \label{fig:FAbehavior}
\end{figure}

\begin{table}
    \centering
    \begin{tabular}{|c|c|c|c|c|c|c|c|c|c|c|}\hline
        Level  & 1 & 2 & 3 & 4 & 5 & 6 & 7 & 8 & 9 \\\hline
        Samples & 776 & 330 & 180 & 134 & 90 & 68 & 52 & 44 & 33\\\hline\hline
        Level & 10 & 11 & 45 & 46 & 47 & 63 & 65 & 119 & 300\\\hline
        Samples & 21 & 22 & 94 & 2 & 2 & 11 & 3 & 11 & 12\\\hline
    \end{tabular}
    \caption{The levels and sampling numbers from a fairly typical trial for a nuclear norm estimation problem. The multilevel method used 15,070 matvecs given a budget of 15,000. }
    \label{fig:level_example}
\end{table}

We then compare the performance of the multilevel method with automated level selection against the single-level estimator, as well as the multilevel estimator with two different sets of prescribed levels. The first of these uses the three levels \{3, 30, 300\}, the sort of selection one might make with no other knowledge of the system. The second uses the seventeen levels \{1,2,\ldots,15, 29, 300\}. This latter choice was informed by using automated selection on a pilot of 100 samples, so we expect it to be reasonably close to the optimal choice for this problem. 

Results are shown in Figure \ref{fig:level_ests}, where 100 trials are run for each method. As expected, the 17-level method is the most accurate with a standard error of approximately 0.47. Automated level selection performs about as well as the 3-level method, with a standard error of approximately 0.54. The single-level approximation has a standard error of about 1.53, lagging significantly behind all of the multilevel variants. 

These results suggest that choosing the levels on the basis of a pilot sample can work well in practice despite the variation in exactly which levels are chosen. It does not appear to be necessary to choose too many levels, as even the three-level method showed significant improvement over the single-level method. In practice, we recommend erring on the side of using too few levels rather than too many since taking a larger number of samples at each level will make the variance estimates more accurate.

\begin{figure}[H]
    \centering
    \includegraphics[width=0.5\textwidth]{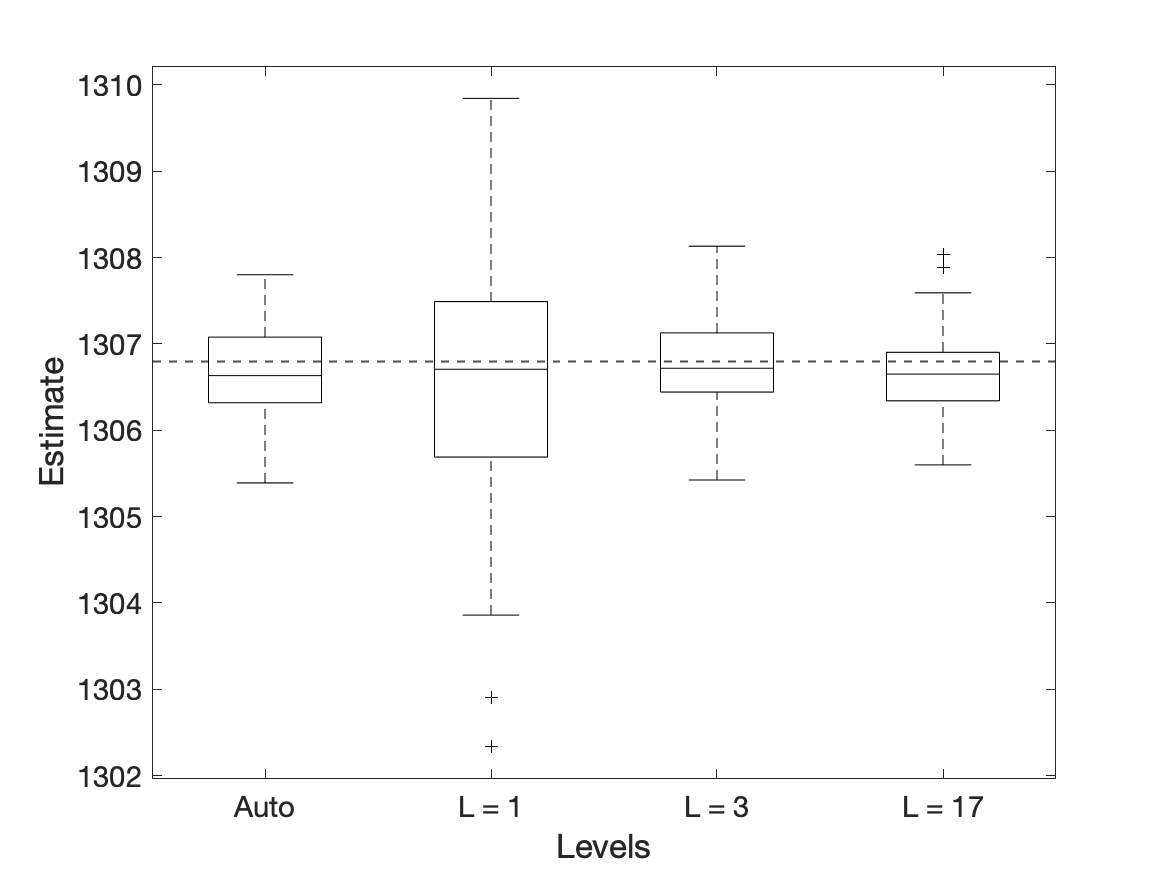}
    \caption{Nuclear norm estimates over 100 trials, comparing automated level selection with using a fixed set of levels.}
    \label{fig:level_ests}
\end{figure}

\subsection{Degree of approximating polynomial}
When using stochastic trace estimation, one faces the problem of deciding how to set the degree of the approximating polynomial $p_n$. Ideally, the degree $n$ and number of samples $m$ should be chosen so that the errors $|\Gamma_m - \trace(p_n(\mat{A}))|$ and $|\trace(p_n(\mat{A})) - \trace(f(\mat{A}))|$ are similar in magnitude\textemdash a large difference between the two suggests a waste of computation, either from drawing too many samples or from using too accurate a polynomial approximation. 

In this experiment, we explore the behavior of single-level and multilevel methods as the degree of the approximating polynomial changes. We again estimated the nuclear norm of the matrix FA, this time allowing the degree $n$ to range from 25 to 350. The single-level method used 50 samples for each trial, and the multilevel method used the equivalent computational budget (i.e., $50n$ matvecs for a degree $n$ polynomial). 

Results are shown in Figure \ref{fig:FAdegreeTest}, where we ran 100 trials for each method and each polynomial degree $n$. The plot on the left shows the approximate standard errors at each degree, as well as the error $|\trace(p_n(\mat{A})) - \trace(f(\mat{A}))|$ due to the polynomial approximation for reference. For the single-level method this quantity is essentially constant, which is to be expected since we take the same number of samples at each degree. The multilevel method outperforms the single-level method even on the coarsest approximation, and continues improving as the approximation degree increases. The reason for this is that as the computational budget increases, the multilevel method devotes most of its effort to taking more samples at the lower levels. The single-level method, by contrast, takes the same number of samples as before but just at higher degrees. 

The plot on the right shows the median relative approximation errors over 100 trials, along with the 25th and 75th percentile errors. For smaller degrees, the accuracy of both single-level and multilevel methods is constrained by the accuracy of the polynomial approximation rather than the number of samples. Somewhere between $n=150$ and $n=200$, the single-level method becomes constrained by the number of samples and shows no further improvement as the degree increases. The multilevel method remains close to optimal until around $n=250$ and continues to improve afterwards. 

One implication of these results is that the advantage of using multilevel methods will be greater when more accurate estimates are desired (and therefore, when higher-degree polynomial approximations are needed). A second implication is that 
multilevel methods are significantly less sensitive to the choice of degree than single-level methods, whose cost for a fixed number of samples grows proportionally to $n$. Various authors \cite{dudley2020monte,ubaru2017fast,han2017approximating} derive error bounds for specific functions that make recommendations for the degree $n$ and number of samples $m$. Although these theoretical bounds are not necessarily tight (particularly for $m$), our results suggest that when using multilevel methods there is little downside to choosing $n$ conservatively.

\begin{figure}[H]
    \centering
    \begin{minipage}{0.49\textwidth}
		\includegraphics[width=\textwidth]{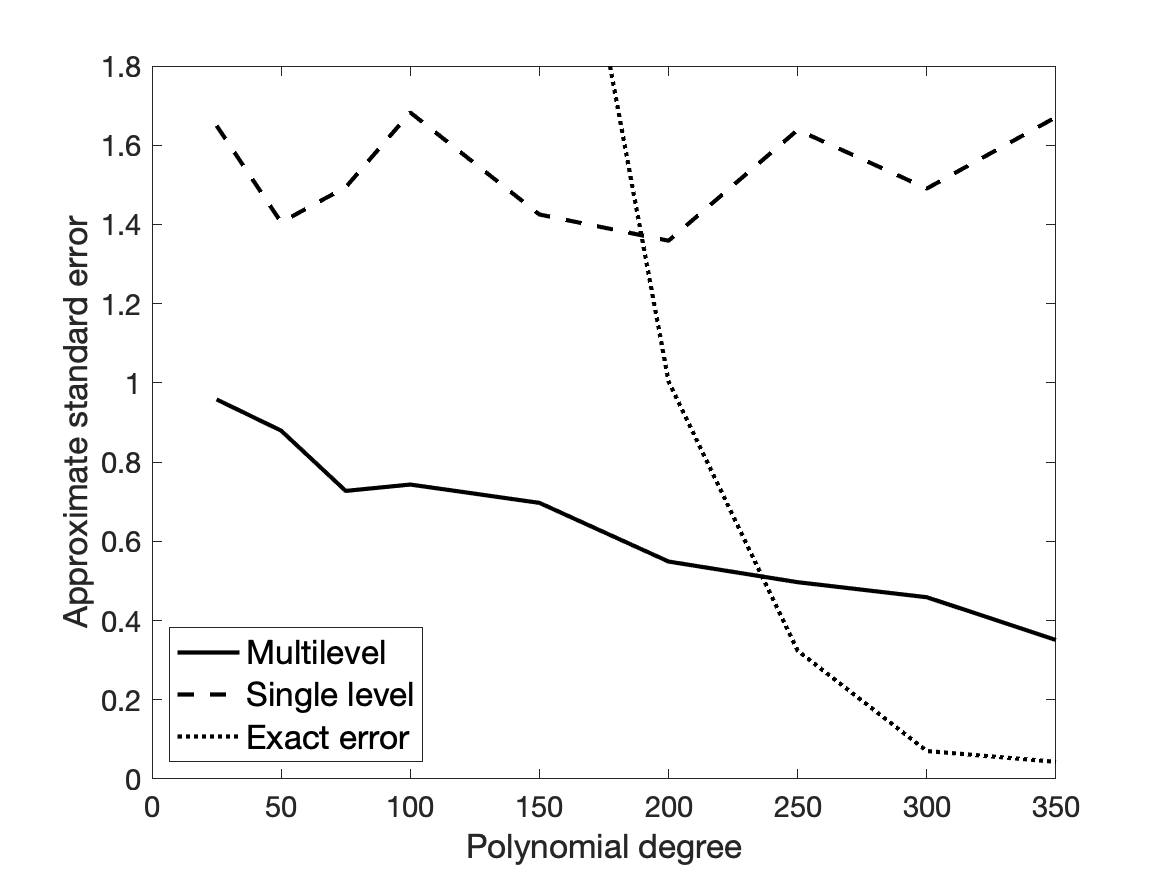}
	\end{minipage}
	\begin{minipage}{0.49\textwidth}
		\includegraphics[width=\textwidth]{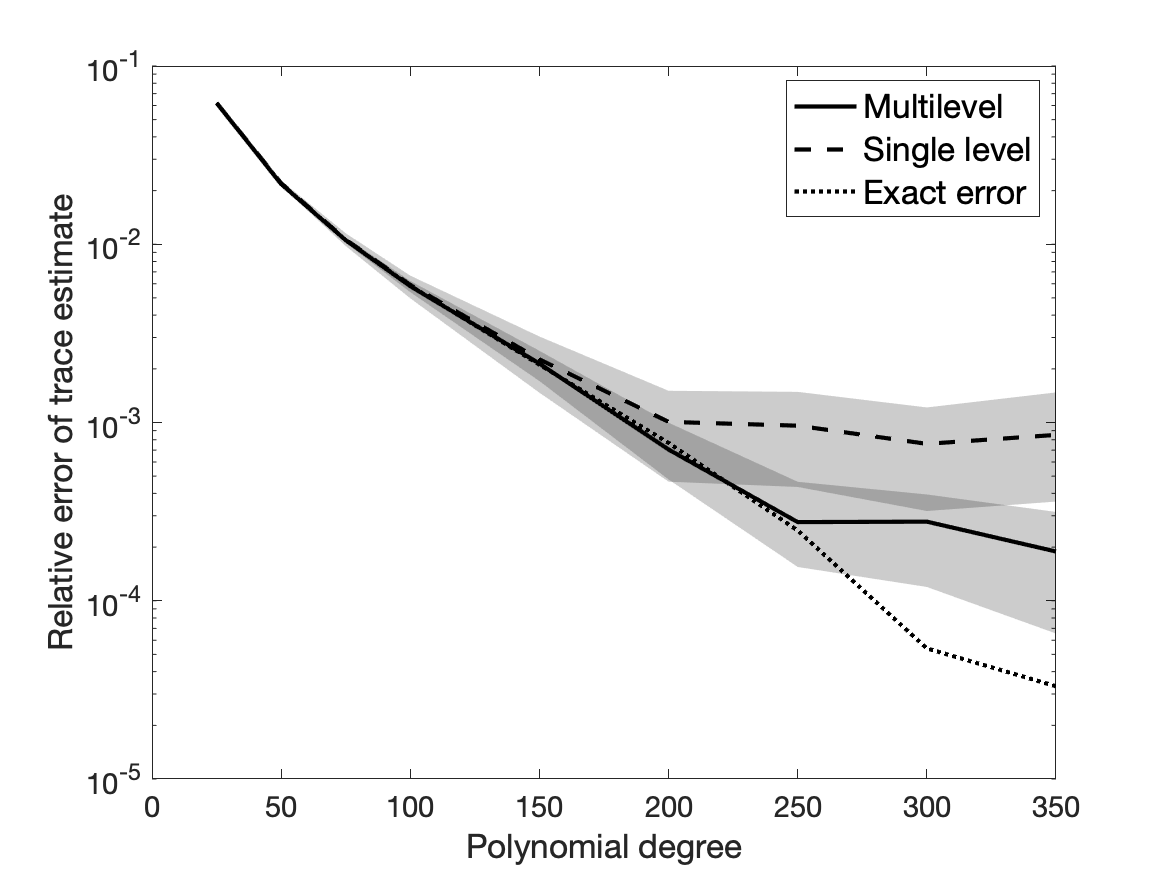}
	\end{minipage}
    \caption{Performance as the approximation degree $n$ changes. Left: approximate standard errors. Right: median relative errors over 100 trials. }
    \label{fig:FAdegreeTest}
\end{figure}





\subsection{SuiteSparse test cases}
Here we show results for the multilevel method and single-level method on a variety of test cases drawn from the SuiteSparse matrix collection. We generally set the degree $n$ large enough to allow for 3-4 digits of precision in the estimate. 

The most promising results for the multilevel method are when estimating the nuclear norm, shown in Figure \ref{fig:nuclear_real}. The singular values of all of these test matrices are available in the SuiteSparse database, and the exact norms are computed using these values. For matrices with low numerical rank such as California, we follow the procedure recommended in \cite{ubaru2017fast} and compute the nuclear norm of $\mat{A}\ts\mat{A}+\lambda \mat{I}$, where $\lambda$ is a small regularization term. This procedure does not change the norm by much, but it does circumvent the problem of the square root function being nondifferentiable at $x=0$. 
The single-level method takes 50 samples for each test case, and given an equivalent computational budget our multilevel method delivers estimates whose standard errors are smaller by a factor of 2.5-4.5. Since the accuracy of an estimate scales with the square root of the number of samples, these results suggest that a multilevel approach could deliver estimates of quality comparable to the single level method while lowering the cost by as much as an order of magnitude.

\begin{table}
    \centering
    \begin{tabular}{|l|c|c|c|c|c|c|}\hline
        \multirow{2}{*}{Matrix} & \multirow{2}{*}{Exact norm} & \multirow{2}{*}{$n$} & \multicolumn{2}{|c|}{Multilevel} & \multicolumn{2}{|c|}{Single Level}\\\cline{4-7}
        & & & Estimate & std & Estimate & std\\\hline
        California & 3803.74 & 100 & 3800.78 & 3.46 & 3802.04 & 10.82\\
        FA & 1306.80 & 300 & 1306.55 & 0.44 & 1305.26 & 1.48\\
        Erdos02 & 3478.23 & 100 & 3481.65 & 5.03 & 3492.99 & 15.31\\
        fe\_4elt2 & 22677.4 & 70 & 22677.1 & 6.68 & 22726.3 & 30.05\\
        deter3 & 16518.1 & 70 & 16514.5 & 3.19 & 16501.0 & 11.45 \\
        uberke1 & 7641.44 & 20 & 7637.79 & 4.69 & 7620.45 & 11.75  \\\hline
    \end{tabular}
    \caption{Nuclear norm estimates with $m = 50$ and $m_{\text{pilot}} = 10$.}
    \label{fig:nuclear_real}
\end{table}

For the test cases estimating the log determinant (Figure \ref{fig:logdet_real}), the exact values are taken from \cite{boutsidis2017randomized}, in which the values are computed using a Cholesky factorization. Here, results are somewhat more modest--the single level method uses 100 samples for each test case, and for the same computational budget the multilevel method gives estimates whose standard errors are smaller by a factor of 1.5-3.5. 

\begin{table}
    \centering
    \begin{tabular}{|l|c|c|c|c|c|c|}\hline
        \multirow{2}{*}{Matrix} & \multirow{2}{*}{Exact logdet} & \multirow{2}{*}{$n$} & \multicolumn{2}{|c|}{Multilevel} & \multicolumn{2}{|c|}{Single Level}\\\cline{4-7}
        & & & Estimate & std & Estimate & std\\\hline
        thermomechTC & -546787 & 75 & -546784 & 9.36 & -546805 & 30.9 \\
        boneS01 & 1.1039e6 & 150 & 1.1040e6 & 25.7 & 1.1039e6 & 77.4\\
        ecology2 & 3.3943e6 & 60 & 3.3933e6 & 158 & 3.3935e6 & 229 \\
        thermal2 & 1.3869e6 & 100 & 1.3864e6 & 182 & 1.3870e6 & 266 \\\hline
    \end{tabular}
    \caption{Log-determinant estimates with $m = 30$ and $m_{\text{pilot}} = 5$.}
    \label{fig:logdet_real}
\end{table}

For the test cases estimating the Estrada index ($f(x) = \exp(x)$, shown in Figure \ref{fig:estrada_real}), the exact values are computed directly. Here, the multilevel method shows little to no improvement over the single-level method. At least part of the reason is that the spectra of these matrices are typically contained in a small interval, and so a small degree $n$ suffices to approximate the exponential function to high accuracy. We observe that the multilevel method typically uses just two levels in this case, the smaller of which was generally around $n/2$. Since the ratio between the lowest and highest levels is small, the multilevel method had little chance to improve over the single-level method. 

\begin{table}
    \centering
    \begin{tabular}{|l|c|c|c|c|c|c|}\hline
        \multirow{2}{*}{Matrix} & \multirow{2}{*}{Exact index} &  \multirow{2}{*}{$n$} & \multicolumn{2}{|c|}{Multilevel} & \multicolumn{2}{|c|}{Single Level}\\\cline{4-7}
        & & & Estimate & std & Estimate & std\\\hline
        fe\_4elt2  & 2.2737e5 & 15 & 2.272e5 & 5.88e2 & 2.261e5 & 8.89e2 \\
        Erdos02  & 1.6705e11 & 20 & 2.206e11 & 2.31e10 & 2.303e11 & 2.50e10\\
        Roget & 2.3797e5 & 20 & 2.113e5 & 1.53e4 & 2.378e5 & 2.37e4\\\hline
    \end{tabular}
    \caption{Estrada index estimates with $m=100$ and $m_{\text{pilot}} = 10$.}
    \label{fig:estrada_real}
\end{table}

In the case of the Estrada index, we also note that the standard errors for our estimates are quite large. This is because the Estrada index of a matrix is dominated by its largest eigenvalues, to a far greater extent than the nuclear norm or log determinant. As a result, it will be particularly helpful to apply the variance reduction methods of \cite{meyer2021hutch++} when estimating the Estrada index. 

\subsection{Graph triangle counting} \label{sec:triangle}
If $\mat{A}$ is the adjacency matrix for an undirected graph, the number of triangles in the graph is known to be equal to $\trace(\mat{A}^3)/6$. We could apply multilevel techniques to estimate this quantity, but it is simpler to just use a control variate instead. For any real numbers $a_1$ and $a_2$, we have that 
\begin{align*}
        \trace(\mat{A}^3) &= \trace(\mat{A}^3 -a_2\mat{A}^2 - a_1\mat{A}) + a_1 \trace(\mat{A}) + a_2\trace(\mat{A}^2)\\
            &= \mathbb{E}[\mat{z}\ts(\mat{A}^3 -a_2\mat{A}^2 - a_1\mat{A})\mat{z}] + a_2 \nnz(\mat{A}).
\end{align*}

The quantities $a_1$ and $a_2$ can then be chosen to minimize the standard deviation of $\mat{z}\ts(\mat{A}^3 -a_2\mat{A}^2 - a_1\mat{A})\mat{z}$. These quantities could be chosen {\it a priori} using the Chebyshev expansion of $x^3$, but for our experiments we compute and store the values $\mat{z}^{(i)T}\mat{A}^j\mat{z}^{(i)}$ for $1\leq i \leq m$ and $1\leq j \leq 3$, then find $a_1$ and $a_2$ through linear regression. The added cost is minimal\textemdash in particular, no extra matvecs with $\mat{A}$ are required. 

We test this variance reduction method on ca-GrQc and wiki-Vote, two standard test graphs. Results are shown in Figure \ref{fig:triangles}, where we report the median relative error over 100 trials along with the 25th and 75th percentile errors. We find that the benefit of using control variates is fairly modest, typically reducing the relative error by around 30\% in the first case and 20\% in the second. Nonetheless, this method is both simple to implement and inexpensive, so there appears to be little drawback to using it.

\begin{figure}[H]
	\centering
	\begin{minipage}{0.49\textwidth}
		\includegraphics[width=\textwidth]{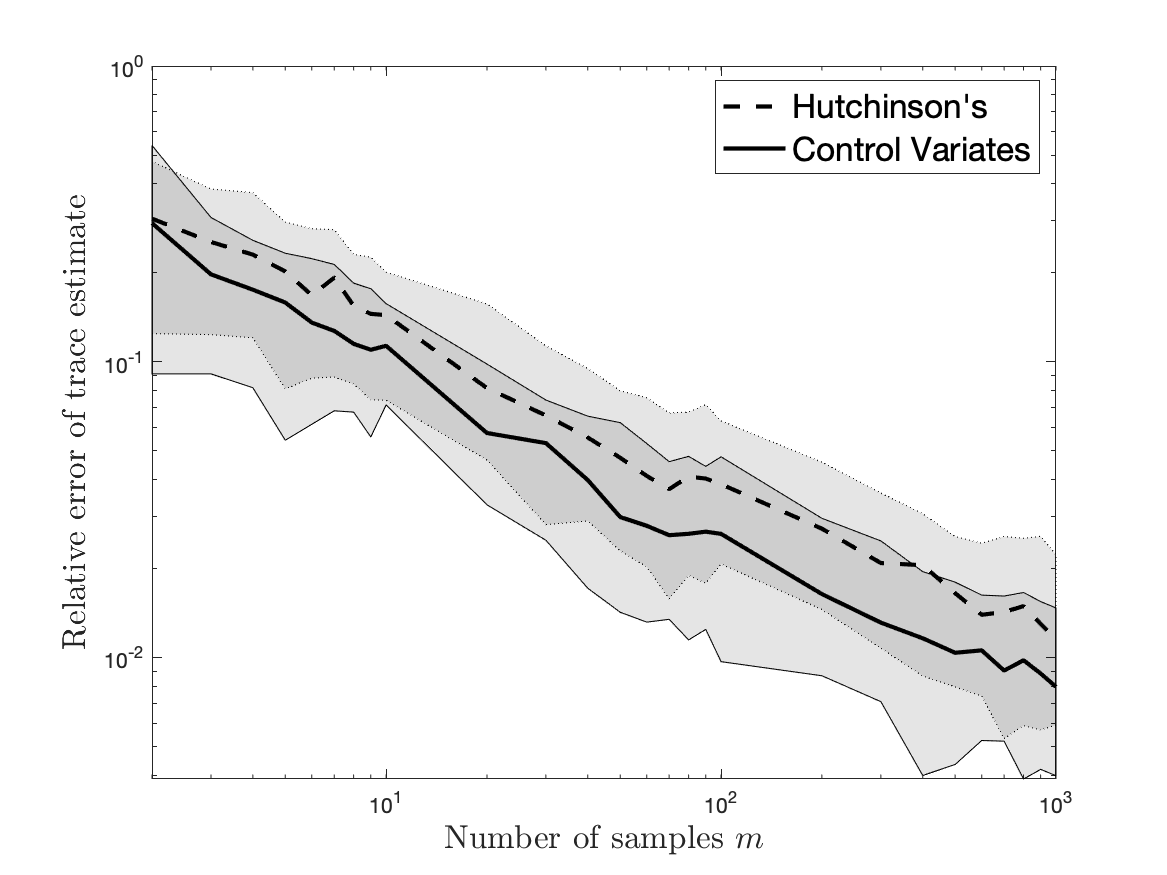}
	\end{minipage}
	\begin{minipage}{0.49\textwidth}
		\includegraphics[width=\textwidth]{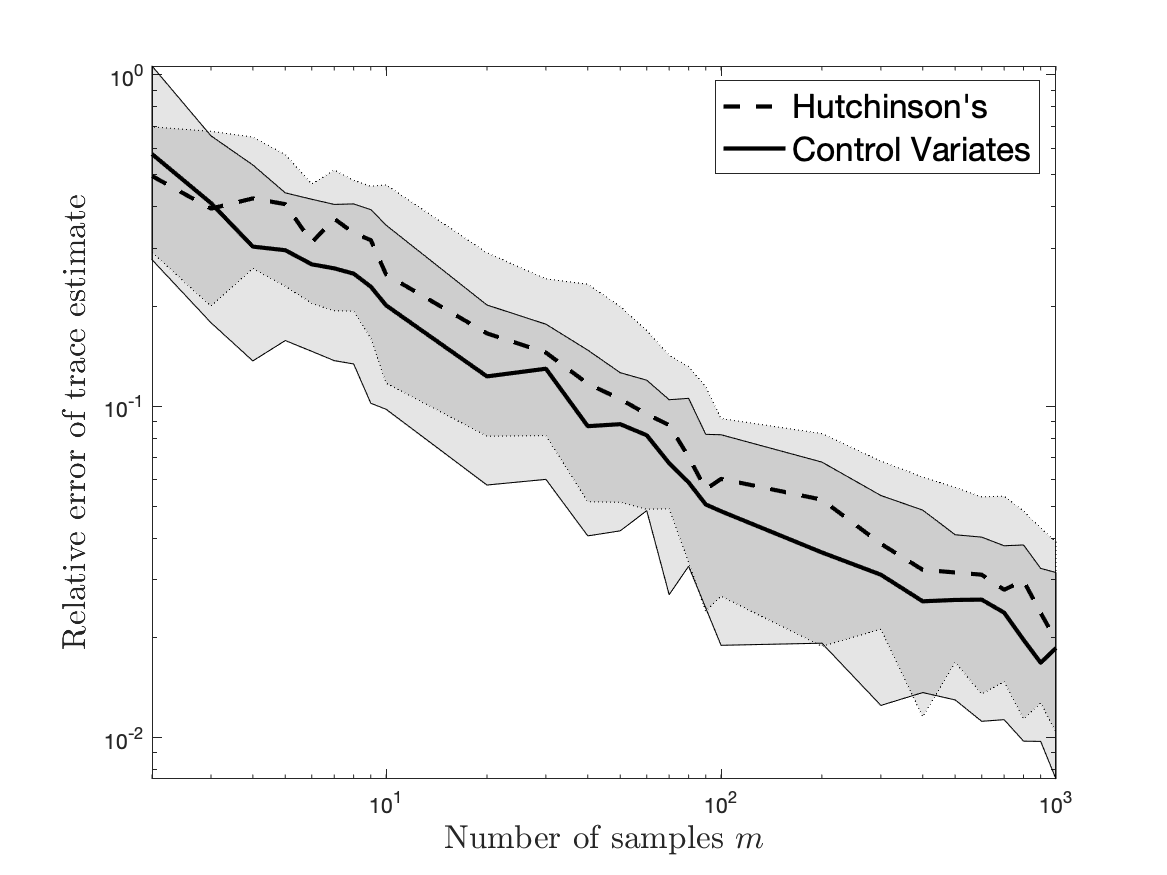}
	\end{minipage}
	\caption{Triangle counting with $f(\mat{A}) = \frac{1}{6}\mat{A}^3$. Left: ca-GrQc, an ArXiv.org collaboration network. Right: wiki-Vote, a Wikipedia administrator voting network.  }
	\label{fig:triangles}
\end{figure}

In theory, we could use these same control variates to estimate the trace of polynomials $p_n(\mat{A})$ of larger degree, such as when approximating the nuclear norm. It is simple to compute $\trace(\mat{A}^j)$ or $\trace(T_j(\mat{A}))$ for $0\leq j \leq 2$, so these low-degree terms may effectively be removed from our variables $Q_{\ell'\ell}$ in \eqref{def:qlevels}. When the degree of the matrix polynomial is large, however, the coefficients of $p_n$ will decay more slowly and so the effect of using control variates will likely be fairly small.

\section{Conclusion} \label{sec:conclusion}
In this paper, we have shown how multilevel techniques can be used to improve existing methods for stochastic trace estimation. We have derived general error bounds for our multilevel trace estimator, and through numerical experiments have demonstrated the efficacy of the multilevel estimator as compared with single-level methods. 

One avenue for further study is in deriving multilevel error guarantees that are specific to the function $f$, such as those for single-level methods in \cite{han2017approximating,ubaru2017fast,dudley2020monte}. Another possibility is to explore whether other variance reduction techniques for Monte Carlo methods might find applications in stochastic trace estimation problems: for example, tools for modeling rare events could potentially be used to determine whether a given matrix is positive definite, a problem which trace estimation is used to solve in \cite{han2017approximating}. Our hope is that this paper will encourage further exploration in these directions. 

\section*{Acknowledgements}
The authors would like to thank Michael Merritt, Alen Alexanderian, and Pierre Gremaud for their helpful remarks. 

\bibliography{refs}
\bibliographystyle{IEEEtran}

\end{document}

%% file: ex_shared.tex
\usepackage[utf8]{inputenc}
\usepackage{amsfonts}
\usepackage{algorithm,algorithmic}
\usepackage{multirow}

\usepackage{lmodern}
\usepackage[T1]{fontenc}

\ifpdf
  \DeclareGraphicsExtensions{.eps,.pdf,.png,.jpg}
\else
  \DeclareGraphicsExtensions{.eps}
\fi

\headers{Multilevel Stochastic Trace Estimation}{E.~Hallman and D.~Troester}

\title{A Multilevel Approach to Stochastic Trace Estimation\thanks{
This research was supported in part by the National Science Foundation through grant DMS-1745654.}}

\author{Eric Hallman\thanks{North Carolina State University
  (\email{erhallma@ncsu.edu}, \url{https://erhallma.math.ncsu.edu/}).}\and Devon Troester }

\newsiamthm{example}{Example}

\theoremstyle{plain}
  \newtheorem{remark}{Remark}

\DeclareMathOperator*{\trace}{tr}

\DeclareMathOperator*{\prob}{Pr}
\DeclareMathOperator*{\nnz}{nnz}

\newcommand{\mat}[1]{{\bf #1}}
\newcommand{\ts}{^T}
\newcommand{\E}{\mathbb{E}}
\newcommand{\var}{\mathbb{V}}